\documentclass{article}
\usepackage{fullpage}
\usepackage{amsmath}
\usepackage{latexsym}
\usepackage{amssymb,amsfonts,amsthm}
\usepackage{graphicx}
\usepackage{graphics}
\usepackage{float}
\usepackage{multirow}
\usepackage{multicol}
\usepackage{hyperref}

\newtheorem{theorem}{Theorem}  

\newtheorem{lemma}[theorem]{Lemma}

\theoremstyle{definition} 

\newtheorem{remark}[theorem]{Remark}

  \newcounter{case}[theorem]

  \newcounter{subcase}[case]

  \newcounter{subsubcase}[subcase]

  \newcounter{subsubsubcase}[subsubcase]

  \newcounter{subsubsubsubcase}[subsubsubcase]

 \def\l{{\lambda}}

 \renewcommand{\-}{\ensuremath{\text{--}}}

\let\oldmarginpar\marginpar
\renewcommand\marginpar[1]{\-\oldmarginpar[\raggedleft\footnotesize #1]%
{\raggedright\footnotesize #1}}

\title{$m$-Cycle Packings of $(\lambda+\mu)K_{v+u}-\lambda K_v$: $m$ even}

\author{John Asplund \\
\texttt{Department of Technology \& Mathematics}\\
\texttt{Dalton State College, GA 30720, USA}\\
\texttt{jasplund@daltonstate.edu}\\}

\begin{document}

\maketitle

\begin{center}
{\bf \footnotesize Abstract}
\end{center}

{\footnotesize 
A $\lambda K_v$ is a complete graph on $v$ vertices with $\lambda$ edges between each pair of the $v$ vertices. A $(\lambda+\mu)K_{v+u}-\lambda K_v$ is a $(\lambda+\mu)K_{v+u}$ with the edge set of $\lambda K_v$ removed. Decomposing a $(\lambda+\mu)K_{v+u}-\lambda K_v$ into edge-disjoint $m$-cycles has been studied by many people. To date, there is a complete solution for $m=4$ and partial results when $m=3$ or $m=5$. In this paper, we are able to solve this problem for all even cycle lengths as long as $u,v\geq m+2$. 
}

\vspace{.3in}

\section{Introduction} \label{intro}

All multigraphs in this paper are loopless. A graph is called \textit{even} if the degree of each vertex is even. 
A $\l K_v$ is a complete graph on $v$ vertices where there are $\l$ edges between each pair of vertices in the graph. If $G_1$ and $G_2$ are graphs such that $G_2$ is a subgraph of $G_1$, then $G_1-G_2$ is the graph $G_1$ with the edge set of $G_2$ removed, e.g. a $(\l+\mu)K_{v+u}-\l K_v$ is a $(\l+\mu)K_{v+u}$ with the edge set of a $\l K_v$ removed.
A \textit{cycle} of length $m$ is denoted as an $m$-cycle. A \textit{decomposition} of a graph $G$ is a partition of the edge set of $G$. A decomposition of a graph $G$ such that each element of the partition is an $m$-cycle is called an \textit{$m$-cycle decomposition} of $G$. Often times we think of the decomposition of a graph as the subgraphs of the elements of the partition. From this context it should be clear which of these definitions we are using. If we want to place emphasis on the set of cycles in the decomposition we will say that $(V,C)$ is an \textit{$m$-cycle system} of $G$ where $V=V(G)$ and $C$ is the set of cycles induced by the partition of $H$ of the edge set of $G$. One instance where we do place emphasis is the set of cycle system enclosings. An $m$-cycle system $(V,C)$ of a graph $G$ is defined where $V$ is the set of vertices in $G$ and $C$ is the set of edge-disjoint $m$-cycles that decomposes the edge set of $G$. 
Alspach {\normalfont\cite{Alspach}} famously conjectured in 1981 that a complete graph could be decomposed into edge-disjoint cycles of arbitrary length. After more than $30$ years, this conjecture was settled in {\normalfont\cite{BHP}}, and not long after the multigraph analogue of this conjecture was settled in {\normalfont\cite{BHMS2}}.  Though investigating the decomposition of a complete graph is a natural starting point, many other types of graphs have been investigated. In particular, we will focus on decomposing a $(\l+\mu)K_{v+u}-\l K_v$ into $m$-cycles.

An \textit{$m$-cycle system} of a graph $G$ is a pair $(V,C)$ where $V$ is the set of vertices in $G$ and $C$ is the set of edge-disjoint $m$-cycles that decomposes the edge set of $G$. An $m$-cycle system $(V,C)$ of $\l K_v$ is said to be \textit{enclosed} in an $m$-cycle system $(V\cup U, C')$ of $(\l+\mu)K_{v+u}$ if $C\subseteq C'$ and $u,\mu\geq 1$. 
There exists an enclosing of an $m$-cycle system of $\l K_v$ in an $m$-cycle system of $(\l+\mu)K_{v+u}$ if and only if there exists both an $m$-cycle decomposition of $(\l+\mu)K_{v+u}-\l K_v$ and an $m$-cycle decomposition of $\l K_v$. 


Only partial results have been shown for the enclosing problem when $m=3$ (see \cite{CHR,HMS,HS,NR2}). In \cite{NR}, the enclosing problem was completely solved in the case when $m=4$. 
Again, only partial results have been shown for the enclosing problem when $m=5$. The necessary conditions for $m=5$ were shown in \cite{AKR} and proved to be sufficient in the cases when $\mu=0$ (technically these are called embeddings), $u=1$, or $u=2$ in \cite{A,AKR,AKR2}.

We aim to show that there exists an $m$-cycle decomposition of $(\l+\mu)K_{v+u}-\l K_v$ when $m$ is even. To do this, we first need to discuss packings and paths. 
A \textit{path of length $k$} is called a $k$-path and is denoted $[a_0,a_1,\ldots,a_{k}]$.
An \textit{$m$-cycle packing} of $G$ is a decomposition of a subgraph $H$ of $G$ into edge-disjoint $m$-cycles. The set of edges in $G$ that are not a part of $H$, that is $E(G)\setminus E(H)$, is called the \textit{leave} of the packing.  If the leave of an $m$-cycle packing $\mathcal{P}$ of $G$ is empty, then $\mathcal{P}$ is an $m$-cycle decomposition. 

In Section~{\normalfont\ref{prelim}}, we give some necessary conditions for the existence of an $m$-cycle decomposition of $(\l+\mu)K_{v+u}-\l K_v$ as well as results that are useful in proving the main theorem (see Theorem~{\normalfont\ref{mainThm}} in Section~{\normalfont\ref{main}}).

\section{Preliminary Results}\label{prelim}

The following necessary conditions for an $m$-cycle decomposition of $(\l+\mu) K_{v+u}-\l K_v$ are similar to those found in {\normalfont\cite{AKR}} except that the conditions presented in this paper are generalized for $m$-cycle decompositions rather than $5$-cycle decompositions.

\begin{theorem}\label{necessaryConditions1}
Let $\l$, $\mu$, $v$, and $u$ be strictly positive integers and let $m>2$ be an integer. If there exists an $m$-cycle decomposition of $(\l+\mu)K_{v+u}-\l K_v$ $(V\cup U,E)$, then
\begin{enumerate}
\item[$(a)$] $u(\l+\mu)+\mu(v-1)\equiv 0\pmod{2}$;
\item[$(b)$] $v(\l+\mu)+(\l+\mu)(u-1)\equiv 0\pmod{2}$;
\item[$(c)$] $(\l+\mu)\binom{u}{2}+vu(\l+\mu)+\mu\binom{v}{2}\equiv 0\pmod{m}$; and
\item[$(d)$] if $u<m$ then 
$$\left\lfloor \frac{(\l+\mu)\binom{u}{2}}{u-1}\right\rfloor (m-u+1)+\varepsilon_1\left(m-\frac{u-1}{2}\right)\leq \mu\binom{v}{2}+vu(\l+\mu)$$ 
where $\varepsilon_1=0$ if $u(\l+\mu)\equiv 0\pmod{2}$ or $\varepsilon_1=1$ if $u(\l+\mu)\equiv 1\pmod{2}$.
\item[$(e)$] if $v<m$ then
$$\left\lfloor \frac{\mu\binom{v}{2}}{v-1}\right\rfloor (m-v+1)+\varepsilon_2\left(m-\frac{v-1}{2}\right)\leq (\l+\mu)\binom{u}{2}+vu(\l+\mu)$$ 
where $\varepsilon_2=0$ if $\mu v\equiv 0\pmod{2}$ or $\varepsilon_2=1$ if $\mu v\equiv 1\pmod{2}$.
\end{enumerate}
\end{theorem}

\begin{proof}
Let $G=(\l+\mu)K_{v+u}-\l K_v$ with vertex set $V\cup U$ such that $\l K_v$ has vertex set $V$. Since the degree of each vertex must be even, Conditions $(a)$ and $(b)$ hold. Since $|E(G)|$ must be even, Condition $(c)$ must hold. 

To show Condition $(d)$ is necessary, let $u<m$. Since $u<m$, each $m$-cycle will contain at most $u-1$ of the $(\l+\mu)\binom{u}{2}$ edges between vertices in $U$. Thus each $m$-cycle must contain at least $(m-u+1)$ edges in $G-(\l+\mu)K_u$.  So we will need at least $\left\lfloor \frac{(\l+\mu)\binom{u}{2}}{u-1}\right\rfloor (m-u+1)$ edges in $G-(\l+\mu)K_u$ to decompose all of the edges in $(\l+\mu)K_u$. In addition, we will also need one $m$-cycle if $(\l+\mu)\binom{u}{2}/(u-1)$ is not an integer, that is, if $u(\l+\mu)\equiv 1\pmod{2}$. This possible extra $m$-cycle will contain at least $m-(u-1)/2$ edges in $G-(\l+\mu)K_u$. 
Hence $(d)$ follows since the right hand side of the inequality is the total number of edges in $G-(\l+\mu)K_u$.
The argument for Condition $(e)$ follows in the same manner as Condition $(d)$.
\end{proof}

%
%



The following result by Bryant \emph{et al.} is the solution to the multigraph analogue of a conjecture by Alspach. It is also one of several results needed for our main theorem. 
\begin{theorem}\label{multiLKn}
{\normalfont\cite{BHMS2}} There is a decomposition $\{G_1,G_2,\ldots,G_t\}$ of $\l K_n$ in which $G_i$ is an $m_i$-cycle for $i=1,2,\ldots,t$ if and only if
\begin{itemize}
\item $\l(n-1)$ is even;
\item $2\leq m_1,m_2,\ldots,m_t\leq n$
\item $m_1+m_2+\cdots+m_t=\l \binom{n}{2}$;
\item $\max(m_1,m_2,\ldots,m_t)+t-2\leq \frac{\l}{2}\binom{n}{2}$ when $\l$ is even; and
\item $\sum_{m_i=2}m_i\leq (\l-1)\binom{n}{2}$
\end{itemize}
\end{theorem}

Let $M=m_1,m_2,\ldots,m_t$ be a sequence of positive integers. A decomposition (packing) of a graph $G$ into $t$ cycles of lengths $m_1,m_2,\ldots,m_t$ is denoted as an $(M)$-cycle decomposition ($(M)$-cycle packing) of $G$. 

This next result is the multigraph analogue of a conjecture by Alspach for complete bipartite multigraphs, though it is not a complete solution.

\begin{theorem}\label{bipartiteMaxPacking}
{\normalfont\cite{JohnJamesJoe}} Let $v$, $u$, and $\l$ be positive integers such that $v,u\geq 5$, $v\leq u$, and $\l v\equiv \l u\equiv 0\pmod{2}$. Let $M=m_1,m_2,\ldots,m_t$ be a sequence of non-decreasing positive even integers such that $2\leq m_1\leq m_2\leq \cdots\leq m_t$. An $(M)$-cycle decomposition of $\l K_{v,u}$ exists if all of the following hold:
\begin{itemize}
\item[$(a)$] $m_t\leq 3m_{t-1}$, 
\item[$(b)$]  $m_{t-1}+m_t\leq 2v+2$ if $v<u$ 
\item[$(c)$] $m_{t-1}+m_t\leq 2v$ if $v=u$, and 
\item[$(d)$] $m_1+ m_2+\cdots+ m_t= \l vu$.
\end{itemize}
\end{theorem}


The next theorem will help us join a complete bipartite graph with a complete graph so that the leaves can be decomposed into two paths with end vertices in the same part, though the result below is stronger than we need. A multigraph analogue of the theorem below will be shown in Section~{\normalfont\ref{main}}. 

\begin{theorem}\label{simpleGraphPaths}
{\normalfont\cite{HH}} Let $m\geq 4$ be an even integer, let $G$ be a complete bipartite graph each of whose parts has even size at least $m+2$, let $R$ be a part of $G$, and let $\ell$ be an integer in $\{4,6,8,\ldots,2m-4\}$ such that $|E(G)|\equiv \ell\pmod{m}$. If $p$ and $q$ are positive even integers such that $p,q\geq \ell-m$ and $p+q=\ell$, then there is an $m$-cycle packing of $G$ whose leave has a decomposition into a $p$-path and a $q$-path such that both end vertices of the paths are in $R$.
\end{theorem}

\section{$m$-cycle packings of $(\l+\mu)K_{v+u}-\l K_v$: $m$ even}\label{main}

To build an $m$-cycle decomposition of $(\l+\mu)K_{v+u}-\l K_v$, we will first need to show how to join two $m$-cycle decompositions to build larger $m$-cycle decompositions. Lemmas~{\normalfont\ref{multiGraphPaths}} and {\normalfont\ref{lemma6.1}} contain these results. Before any of this, we need several ancillary results. 

\begin{theorem}\label{lemma3.4}
{\normalfont\cite{JohnJamesJoe}} Suppose that there exists an $(M)$-cycle packing of $\l K_{v,u}$ with a leave $L$. If $a$ and $b$ are vertices in the same part of $\l K_{v,u}$ such that $\deg_L(a)>\deg_L(b)$, then there exists an $(M)$-cycle packing of $\l K_{v,u}$ with a leave $L'$ such that $\deg_{L'}(a)=\deg_L(a)-2$, $\deg_{L'}(b)=\deg_L(b)+2$, and $\deg_{L'}(x)=\deg_L(x)$ for all $x\in V(L)\setminus\{a,b\}$. Furthermore, this $L'$ also satisfies
\begin{itemize}
\item[$(i)$] if $\deg_L(b)=0$ and $a$ is not a cut vertex of $L$, then $L'$ has the same number of non-trivial components as $L$; and
\item[$(ii)$] if $\deg_L(b)=0$, then either $L'$ has the same number of non-trivial components as $L$ or $L'$ has one more non-trivial component than $L$.
\end{itemize}
\end{theorem}


A \emph{chain} is a collection of cycles $A_1,A_2,\ldots,A_r$ such that
\begin{itemize}
\item $A_i$ is a cycle of length $2\leq a_i$, and 
\item for $1\leq i<j\leq r$, $|V(A_i)\cap V(A_j)|=1$ if $j=i+1$ and $|V(A_i)\cap V(A_j)|=0$ otherwise.
\end{itemize}
The cycles $A_1$ and $A_r$ are called \emph{end cycles}, the cycles $A_2,A_3,\ldots,A_{r-1}$ are called \emph{internal cycles}, and the vertex in $A_i\cap A_{i+1}$ is called the \emph{link vertex}. A chain containing $r$ cycles is called an \emph{$r$-chain}. A $2$-chain with cycles $C_1$ and $C_2$ will be denoted $C_1\cdot C_2$ or as a $(c_1,c_2)$-chain where $c_1$ and $c_2$ are the lengths of $C_1$ and $C_2$ respectively. 

In the following theorem, we think of the leave as being a graph with $V=V(G)$ and $E=L$.

\begin{theorem}\label{lemma3.5}
{\normalfont\cite{JohnJamesJoe}} Suppose that there exists an $(M)$-cycle packing $\mathcal{P}_0$ of $\l K_{v,u}$, where $v\leq u$, with a leave $L_0$ of size $\ell$, where $\ell\leq 2v+2$ if $v< u$ and $\ell \leq 2v$ if $v=u$, with $k_0$ non-trivial components such that $L_0$ has at least one vertex of degree at least $4$. Then, there exists an $(M)$-cycle packing of $\l K_{v,u}$ with a leave $L'$ such that exactly one vertex of $L'$ has degree $4$, every other vertex of $L'$ has degree $2$ or degree $0$, and $L'$ has at most $\min(\{k_0+d(\mathcal{P}_0)-1, \lfloor \frac{\ell}{2}\rfloor-1\})$ non-trivial components where
$$d(\mathcal{P}_0)=\frac{1}{2}\sum_{x\in D}(\deg_L(x)-2),$$
where $D$ is the set of vertices of $L$ having degree at least $4$.
\end{theorem}

\begin{theorem}\label{lemma3.2}
{\normalfont\cite{JohnJamesJoe}} Suppose that there exists an $(M)$-cycle packing $\mathcal{P}$ of $\l K_{v,u}$ with a leave $L$ of size $\ell$ with $k$ non-trivial components such that exactly one vertex of $L$ has degree $4$ and every other vertex of $L$ has degree $2$ or degree $0$. If 
$m_1$ and $m_2$ are integers such that $m_1,m_2\geq k+1$ and $m_1+m_2=\ell$, then there exists an $(M)$-cycle packing of $\l K_{v,u}$ with a leave whose only non-trivial component is a chain which has a decomposition into an $m_1$-path and an $m_2$-path.
\end{theorem}

This next result will allow us to join an $m$-cycle decomposition of $\l K_{v,u}$ or $(\l+\mu)K_{v,u}$ to an $m$-cycle decomposition of $\l K_u$ or $(\l+\mu)K_u$ respectively in our main theorem. 

\begin{lemma}\label{multiGraphPaths}
Let $\l$ be a positive integer. Let $m\geq 4$ be an even integer if $\l=1$ and let $m\geq 2$ be an even integer if $\l>1$.  Let $G$ be a complete bipartite multigraph with multiplicity $\l$ where each part has size at least $m+2$, and each vertex has even degree.  Let $\ell\in\{4,6,8,\ldots,2m-4\}$ if $\l=1$ or $\ell\in\{2,4,6,\ldots,2m-2\}$ if $\l\geq 2$ such that in either case $|E(G)|\equiv \ell\pmod{m}$. If $p$ and $q$ are positive even integers such that $p,q\geq \ell-m$ and $p+q=\ell$, then there is an $m$-cycle packing of $G$ whose leave has a decomposition into a $p$-path and a $q$-path.
\end{lemma}

\begin{proof}
If $\l=1$, the result follows by Theorem~{\normalfont\ref{simpleGraphPaths}}. So now assume that $\l\geq 2$.

If $\ell\leq m+2$ then $\ell\leq 3m$ (Condition $(a)$ of Theorem~{\normalfont\ref{bipartiteMaxPacking}}), $2(m+2)\leq 2\min(\{u,v\})\leq 2(v+u)$ (Conditions $(b)$ and $(c)$ of Theorem~{\normalfont\ref{bipartiteMaxPacking}}), and $m+m+\cdots+m+\ell=\l uv$ (Condition $(d)$ of Theorem~{\normalfont\ref{bipartiteMaxPacking}}), so there exists an $m$-cycle packing of $G$ by Theorem~{\normalfont\ref{bipartiteMaxPacking}} with a leave that is a single cycle of length $\ell$. Then it is clear that there exists a decomposition of the leave to form the required $p$-path and $q$-path. Thus we suppose that $\ell\geq m+4$. 

Again, by Theorem~{\normalfont\ref{bipartiteMaxPacking}}, there exists an $m$-cycle packing $\mathcal{P}$ of $G$ with a leave whose only non-trivial component is an $(\ell-m)$-cycle. It is clear that there is an $m$-cycle that shares at least one vertex in common with the $(\ell-m)$-cycle. By removing this cycle from the packing $\mathcal{P}$ we form a new packing $\mathcal{P}_1$ and leave $L_1$ in which $L_1$ contains exactly one non-trivial component composed of an $(\ell-m)$-cycle and an $m$-cycle which share at least one vertex and at most $\ell-m$ vertices in common. That is, the leave contains at least one vertex of degree $4$, and all other vertices are either degree $2$ or degree $0$. The leave $L_1$ must contain at least one vertex of degree $0$ since the $(\ell-m)$-cycle and the $m$-cycle in the leave share one vertex in common and so $\frac{(\ell-m)+m}{2}-1=\frac{\ell-2}{2}\leq \frac{2m-4}{2} \leq m-2\leq \min(\{u,v\})$, that is, the non-trivial component of $L_1$ does not contain all vertices in either part of the partition of $G$. Also, notice that if $L_1$ contains $\ell-m$ vertices of degree $4$, then $L_1$ contains no cut vertex since the only non-trivial component of $L_1$ contains exactly $m$ vertices and an $m$-cycle. 
Thus by applying Theorems~{\normalfont\ref{lemma3.4}} and {\normalfont\ref{lemma3.5}} we will have at most $\ell-m-1$ non-trivial components. From the previous observations, it follows that by applying Theorem~{\normalfont\ref{lemma3.5}}, we can form an $m$-cycle packing of $G$ whose leave $L'$ contains exactly one component with a vertex of degree $4$ and at most $\ell-m-1$ non-trivial components. 
Thus by Theorem~{\normalfont\ref{lemma3.2}}, we can form an $m$-cycle packing of $G$ whose leave can be decomposed into a $p$-path and a $q$-path.
\end{proof}


\begin{theorem}\label{setsAndRelabeling}
{\normalfont\cite{horsley2012decomposing}} Let $A$ be a set, let $S$ and $T$ be subsets of $A$, and let $s'$ and $t'$ be non-negative integers. Then, there exist subsets $S'$ and $T'$ of $A$ such that $|S'|=s'$, $|T'|=t'$, $S\cap S'=\varnothing$, $T\cap T'=\varnothing$, and $S'\cap T'=\varnothing$ if and only if 
\begin{enumerate}
\item[$(i)$] $|S\cap T|+s'+t'\leq |A|$;
\item[$(ii)$] $|S|+s'\leq |A|$; and 
\item[$(iii)$] $|T|+t'\leq |A|$.
\end{enumerate}
\end{theorem}


\begin{remark}
Though the result in Theorem~{\normalfont\ref{setsAndRelabeling}} was originally stated with $s'$ and $t'$ as positive integers, the result holds when either $s'$ or $t'$ is $0$. 
\end{remark}

Following the terminology in {\normalfont\cite{horsley2012decomposing}}, we define the triple $(A,S,T)$ to be $(s',t')$-good if $A$, $S$, and $T$ are sets and both $s'$ and $t'$ are integers such that $S,T\subseteq A$ and Conditions $(i)$, $(ii)$, and $(iii)$ of Theorem~{\normalfont\ref{setsAndRelabeling}} hold. 

The following lemma is a multigraph analogue of Lemma~6.1 in {\normalfont\cite{horsley2012decomposing}}. Though the proof provided in Lemma~6.1 of {\normalfont\cite{horsley2012decomposing}} was sufficient to prove the next lemma with generalizations, their lemma did not include that both end vertices of the paths in the leave are in $B$, which is necessary to our main theorem. As such, a complete proof to Lemma~{\normalfont\ref{lemma6.1}} is provided below.


\begin{lemma}\label{lemma6.1}
Let $H_2$ be an even complete bipartite multigraph with parts $A$ and $B$ and let $H_1$ be an even multigraph with vertex set $A$. Let $H_3$ be an even graph on vertex set $A\cup\{\infty\}$ where $\infty\not\in A\cup B$. Suppose there exists an $(M_1)$-cycle packing $\mathcal{P}_1$ of $H_1$ whose leave has a decomposition into a $p$-path $P_1$ and a $q$-path $Q_1$, there exists an $(M_2)$-cycle packing $\mathcal{P}_2$ of $H_2$ with a leave whose only non-trivial component is an $\ell$-cycle, and there exists an $(M_3)$-cycle packing $\mathcal{P}_3$ of $H_3$ whose leave can be decomposed into a $p$-path $P_3$ and a $q$-path $Q_3$ with both end vertices in $A$.
Then for $(H^*,P^*,Q^*,\mathcal{P}^*,M^*)\in\{(H_1,P_1,Q_1,\mathcal{P}_1,M_1),(H_3,P_3,Q_3,\mathcal{P}_3,M_3)\}$ the following hold.
\begin{enumerate}
\item[$(a)$] For any even integers $p'$ and $q'$ such that $p',q'\geq 2$ and $p'+q'=\ell$, if $(A,V(P^*),V(Q^*))$ is $\left(\frac{p'-2}{2},\frac{q'-2}{2}\right)$-good, then there exists an $(\widetilde{M})$-cycle decomposition of $H^*\cup H_2$ where $\widetilde{M}$ is the sequence $M^*,M_2,p+p',q+q'$.
\item[$(b)$] For any even integers $p'$ and $q'$ such that $p',q'\geq 2$ and $p'+q'=\ell$, if $(A,V(P^*),V(Q^*))$ is $\left(\frac{p'-2}{2},\frac{q'-2}{2}\right)$-good, then there exists an $(\widetilde{M})$-cycle packing of $H^*\cup H_2$ where $\widetilde{M}$ is the sequence $M^*,M_2$, whose leave has a decomposition into a $(p+p')$-path and a $(q+q')$-path with both end vertices in $B$.
\end{enumerate}
\end{lemma}

\begin{proof}
Suppose the assumptions stated in this lemma hold. We will handle parts $(a)$ and $(b)$ separately. 

\textbf{Part (a):} Since $H^*$ is an even graph, the leave of $H^*$ is also an even graph. Let $P^*$ and $Q^*$ have end vertices $a,a'\in A$. Notice that in the case that $P^*=P_3$ and $Q^*=Q_3$ we assumed the end vertices were in $A$. Even though $P_3$ or $Q_3$ may include $\infty$ as one of its vertices, Theorem~{\normalfont\ref{setsAndRelabeling}} can still be used since it requires only sets of vertices rather than the edges of the cycles. By using Theorem~{\normalfont\ref{setsAndRelabeling}} (since $(A,V(P^*),V(Q^*))$ is $\left(\frac{p'-2}{2},\frac{q'-2}{2}\right)$-good) we can relabel the vertices in $\mathcal{P}_2$ so that the leave of $\mathcal{P}_2$ can be decomposed into a $p'$-path and a $q'$-path with end vertices $a$ and $a'$ such that $V(P^*)\cap V(P')=V(Q^*)\cap V(Q')=\{a,a'\}$. Thus we obtain the required decomposition.


\textbf{Part (b):} Since $H^*$ is an even graphs, the leave of $H^*$ is also an even graph. Let $P^*$ and $Q^*$ have end vertices $a,a'\in A$. Let $b$ and $b'$ be distinct vertices in $B$. Again, notice that in the case that $P^*=P_3$ and $Q^*=Q_3$ we assumed the end vertices were in $A$. By using Theorem~{\normalfont\ref{setsAndRelabeling}} (since $(A,V(P^*),V(Q^*))$ is $\left(\frac{p'-2}{2},\frac{q'-2}{2}\right)$-good) we can relabel the vertices in $\mathcal{P}_2$ so that the leave of $\mathcal{P}_2$ can be decomposed into two 1-paths $[a,b]$ and $[a,b']$, a $(p'-1)$-path $P'$ from $a'$ to $b'$, and a $(q'-1)$-path $Q'$ from $a'$ to $b$ such that $V(P')\cap V(P^*)=V(Q')\cap V(Q^*)=\{a'\}$. Then $\mathcal{P}^*\cup \mathcal{P}_2$ is an $(M^*,M_2)$-packing of $H^*\cup H_2$ with a leave that can be decomposed into a $(p+p')$-path $P^*\cup P'\cup [a,b]$ and a $(q+q')$-path $Q^*\cup Q'\cup [a,b']$, attaining the required packing.
\end{proof}

\begin{theorem}\label{mainThm}
Let $\l$ and $\mu$ be integers such that $\l> 0$ and $\mu\geq 0$. Let $v$ and $u$ be positive integers. If $m\geq 4$ is an even integer such that $v\geq m+2$, $u\geq m+2$, and the conditions in Theorem~{\normalfont\ref{necessaryConditions1}} are satisfied, then there is an $m$-cycle decomposition of $(\l+\mu)K_{v+u}-\l K_v$. 
\end{theorem}

\begin{proof}
Let $\l,\mu,v,$ and $u$ satisfy the necessary conditions in Theorem~{\normalfont\ref{necessaryConditions1}}. The case when $m=4$ has been settled in {\normalfont\cite{newman}}, so we may assume $m\geq 6$. We will break this proof into two cases: 1) $\l+\mu$ is even, or both $\l+\mu-1$ and $\l$ are even; and 2) $\l+\mu$ and $\l$ are odd.

{\bf Case 1:} Suppose $\l+\mu$ is even, or both $\l+\mu-1$ and $\l$ are even.

If $\l+\mu$ is even, let $G_1=\mu K_v$ with vertex set $V$, $G_2=(\l+\mu)K_{v,u}$ with vertex set $V\cup U$, $G_3=(\l+\mu)K_u$ with vertex set $U$, and $G'=G_2\cup G_3=(\l+\mu)K_{v+u}-(\l+\mu)K_v$. If $\l+\mu-1$  and $\l$ are even, let $G_1=\mu K_{v+u}$ with vertex set $V\cup U$, $G_2=\l K_{v,u}$ with vertex set $V\cup U$, $G_3=\l K_u$ with vertex set $U$, and $G'=G_2\cup G_3=\l K_{v+u}-\l K_v$. 
Let $\ell_1,\ell_3\in\{0,4,6,8,\ldots,m-2,m+2\}$ such that $\ell_1\equiv |E(G_1)|\pmod{m}$ and $\ell_3\equiv |E(G_3)|\pmod{m}$. 
Let $\ell_2\in\{4,6,8,\ldots,m-4,m-2,m,m+2\}$ such that $\ell_2\equiv |E(G_2)|\pmod{m}$. By Theorem~{\normalfont\ref{multiLKn}} and by our assumptions on $\l$ and $\mu$, there exists an $m$-cycle packing $\mathcal{P}_1$ of $G_1$ and $m$-cycle packing $\mathcal{P}_3$ of $G_3$ with a leave $L_1$ and $L_3$ that contains is a single cycle of length $\ell_1$ and length $\ell_3$ respectively. Note that if $\ell_1=0$ or $\ell_3=0$, then the leave is considered empty.

If $\ell_1=\ell_3=0$, then $\ell_2=m$ since $\ell_1+\ell_2+\ell_3\equiv 0\pmod{m}$. If this were the case, we could form an $m$-cycle decomposition of $G_2$ by Theorem~{\normalfont\ref{bipartiteMaxPacking}}, and thus by joining this decomposition with $\mathcal{P}_1$ and $\mathcal{P}_3$ in the natural way, we form the required packing.

Let $r=\frac{m}{2}$ if $\frac{m}{2}$ is even and let $r=\frac{m}{2}-1$ if $\frac{m}{2}$ is odd. Define $e$, $p_1$, $q_1$, $p_2$, $q_2$, $p_3$, $q_3$, $p_4$, and $q_4$ as in Table~{\normalfont\ref{easier}} depending on $\ell_1$, $\ell_2$, and $\ell_3$.

\begin{table}[htb]
\begin{center}
\begin{tabular}{|c|c|c|c|c|c|c|c|c|c|c|}
\hline
& $e$ & $p_1$ & $q_1$ & $p_2$ & $q_2$ & $p_3$ & $q_3$ & $p_4$ & $q_4$ \\
\hline
\hline
$\ell_1=0$, $\ell_2+\ell_3=m$ & $m+\ell_2$ & $-$ & $-$ & $m-2$ & $\ell_2+2$ & $2$ & $\ell_3-2$ & $-$ & $-$ \\
\hline
$\ell_1=0$, $\ell_2+\ell_3=2m$ & $\ell_2$ & $-$ & $-$ & $r$ & $\ell_2-r$ & $r$ & $\ell_3-r$ & $-$ & $-$ \\
\hline
$\ell_3=0$, $\ell_1+\ell_2=m$ & $m+\ell_2$ & $2$ & $\ell_3-2$ & $m-2$ & $\ell_2+2$ & $-$ & $-$ & $-$ & $-$ \\
\hline
$\ell_3=0$, $\ell_1+\ell_2=2m$ & $\ell_2$ & $r$ & $\ell_3-r$ & $r$ & $\ell_2-r$ & $-$ & $-$ & $-$ & $-$ \\
\hline
$\ell_1\neq 0$, $\ell_3\neq 0$ & \multirow{2}{*}{$m+\ell_2$} & \multirow{2}{*}{$1$} & \multirow{2}{*}{$\ell_1-1$} & \multirow{2}{*}{$m-2$} & \multirow{2}{*}{$\ell_2+2$} & \multirow{2}{*}{$1$} & \multirow{2}{*}{$\ell_3-1$} & \multirow{2}{*}{$m-1$} & \multirow{2}{*}{$\ell_2+\ell_3+1$} \\
$\ell_1+\ell_2+\ell_3=m$ &&&&&&&&& \\
\hline
$\ell_1\neq 0$, $\ell_3\neq 0$  & \multirow{3}{*}{$\ell_2$} & \multirow{3}{*}{$m-3$} & \multirow{3}{*}{$5$} & \multirow{3}{*}{$2$ } & \multirow{3}{*}{$\ell_2-2$} & \multirow{3}{*}{$1$} & \multirow{3}{*}{$\ell_3-1$} & \multirow{3}{*}{$3$} & \multirow{3}{*}{$\ell_2+\ell_3-3$} \\
$\ell_1+\ell_2+\ell_3=2m$, &&&&&&&&& \\
$\ell_2+\ell_3< m+2$ &  &  &  & &  &  &  &  & \\
\hline
\end{tabular}
\end{center}
\caption{Definitions of variables based on $\ell_1$, $\ell_2$, and $\ell_3$}\label{easier}
\end{table}

Suppose that $\ell_1$, $\ell_2$, and $\ell_3$ are defined as in any one row of Table~{\normalfont\ref{easier}} and suppose that if $\ell_1\neq 0$ and $\ell_1+\ell_2+\ell_3=2m$, then $\ell_2+\ell_3<m+2$. If $\ell_1\neq 0$, $\ell_3\neq 0$, $\ell_1+\ell_2+\ell_3=2m$, and $\ell_2+\ell_3<m+2$, then $\ell_1=m+2$ and $\ell_2+\ell_3=m-2$ (since otherwise $\ell_2+\ell_3=m$ and thus $\ell_1=0$, a contradiction).
Recall that $u\geq m+2$ and $v\geq m+2$. Notice that in each case when $e=m+\ell_2$, $\ell_2\leq m-4$ and thus $e\leq 2m-4$; so the requirement in Theorem~{\normalfont\ref{multiGraphPaths}} that the size of the leave is at most $2m-4$ when $v\geq m+2$ is satisfied ($e$ will represent the size of the leave of an $m$-cycle packing of $G_2$).
Then by Theorem~{\normalfont\ref{multiGraphPaths}}, there exists an $m$-cycle packing $\mathcal{P}_2$ of $G_2$ with a leave that is a cycle of size $e$.
It is clear that $L_3$ can be decomposed into a $p_3$-path $P_3$ and an $q_3$-path $Q_3$ (assuming $L_3$ is non-empty). 
Since $|(V(P_3)\cap U)\cap (V(Q_3)\cap U)|\leq \min(\{p_3,q_3\})$, $|V(P_3)\cap U|=p_3+1$, $|V(Q_3)\cap U|=q_3+1$, and by our choices in Table~{\normalfont\ref{easier}}, it follows that the conditions in Theorem~{\normalfont\ref{setsAndRelabeling}} are satisfied, so $(U,V(P_3)\cap U, V(Q_3)\cap U)$ is $(\frac{p_2-2}{2},\frac{q_2-2}{2})$-good. 
If $\ell_1=0$, by  Lemma~{\normalfont\ref{lemma6.1}}(a) (with $p=p_3$, $q=q_3$, $p'=p_2$, and $q'=q_2$), there exists an $m$-cycle decomposition of $G'$ and so along with the packing $\mathcal{P}_1$, we attain the required decomposition of $(\l+\mu)K_{v+u}-\l K_v$. Otherwise, after we join $G_2$ and $G_3$ to form $G'$, we must join the leave in $G'$ and the leave in $G_1$ together. 
Thus by Lemma~{\normalfont\ref{lemma6.1}}(b) (with $p=p_3$, $q=q_3$, $p'=p_2$, and $q'=q_2$), there exists an $m$-cycle packing $\mathcal{P}'$ of $G'$ with a leave that can be decomposed into an $p_4$-path $P'$ and an $q_4$-path $Q'$. 
There is a clear decomposition of $L_1$ into a $p_1$-path $P_1$ with end vertices $x$ and $y$ and a $q_1$-path $Q_1$. Note that since $|(V(P_1)\cap V)\cap (V(Q_1)\cap V)|\leq \min(\{p_1,q_1\})$ and $\ell_2+\ell_3<m-2$ when $\ell_1+\ell_2+\ell_3=2m$ and $\ell_1\neq 0$, it follows that 
\begin{align}\label{ScapTArgument}
\min(\{p_1,q_1\})+\frac{p_4-2}{2}+\frac{q_4-2}{2} &\leq \begin{cases}
1+\frac{m-3}{2}+\frac{\ell_2+\ell_3-1}{2} &\mbox{if $\ell_1+\ell_2+\ell_3=m$, and} \\
5+\frac{3-2}{2}+\frac{\ell_2+\ell_3-3}{2} &\mbox{if $\ell_1+\ell_2+\ell_3=2m$} \\
\end{cases}\\ \notag
&\leq \begin{cases}
1+\frac{m-3}{2}+\frac{m+1}{2} &\mbox{if $\ell_1+\ell_2+\ell_3=m$, and} \\
\frac{11}{2}+\frac{m-1}{2} &\mbox{if $\ell_1+\ell_2+\ell_3=2m$} \\
\end{cases}\\
&\leq m<v-1. \notag
\end{align}
Then by using Theorem~{\normalfont\ref{setsAndRelabeling}} (since $(V, V(P_1)\cap V, V(Q_1)\cap V)$ is $(\left\lfloor\frac{p_4-2}{2}\right\rfloor,\left\lfloor\frac{q_4-2}{2}\right\rfloor)$-good), the vertices of $\mathcal{P}_1$ can be relabeled 
so that $V(Q_1)\cap V(Q')=V(P_1)\cap V(P')=\{x,y\}$. Thus there exists an $m$-cycle decomposition $\mathcal{P}'\cup \mathcal{P}_1\cup \{E(P'\cup P_1),E(Q'\cup Q_1)\}$ of $(\l+\mu)K_{v+u}-\l K_v$. 

Now suppose $\ell_1\neq 0$, $\ell_1+\ell_2+\ell_3=2m$, and $\ell_2+\ell_3\geq m+2$, so $\ell_2+\ell_3\leq 2m-4$. 
By Theorem~{\normalfont\ref{bipartiteMaxPacking}}, there exists an $m$-cycle packing $\mathcal{P}_2$ of $G_2$ with a leave that is a cycle of size $\ell_2$.
Let $s$ and $s'$ be positive integers such that $s+s'=m-1$, $\ell_3-s\geq 1$, $\ell_2-s'\geq 2$, and $s'$ is even. Such integers exist since $\ell_2+\ell_3\geq m+2$ and $\ell_2,\ell_3>0$. It is clear that $L_3$ can be decomposed into an $s$-path $P_3$ and an $(\ell_3-s)$-path $Q_3$.  
Since $|V(P_3)\cap U|=s$, $|V(Q_3)\cap U|=\ell_3-s$ and hence $|(V(P_3)\cap U)\cap (V(Q_3)\cap U)|\leq \frac{\ell_3}{2}$ and $\frac{\ell_3}{2}+\frac{s'-2}{2}+\frac{\ell_2-s'-2}{2}\leq \frac{2m}{2}=m$, it follows by Theorem~{\normalfont\ref{setsAndRelabeling}} that $(U,V(P_3)\cap U, V(Q_3)\cap U)$ is $(\frac{s'-2}{2},\frac{\ell_2-s'-2}{2})$-good. 
Thus by Lemma~{\normalfont\ref{lemma6.1}}(b) (with $p=s$, $q=\ell_3-s$, $p'=s'$, and $q'=\ell_2-s'$), there exists an $m$-cycle packing $\mathcal{P}'$ of $G'$ with a leave that can be decomposed into an $(m-1)$-path $P'$ and an $(\ell_2+\ell_3-m+1)$-path $Q'$ with both end vertices in $V$. 
There is a clear decomposition of $L_1$ into a $1$-path $P_1=[x,y]$ and an $(\ell_1-1)$-path $Q_1$. 
Note that since $|(V(P')\cap V)\cap (V(Q')\cap V)|\leq \min(\{m-1,\ell_2+\ell_3-s-s'\})$, it follows that 
\begin{align}\label{ScapTArgument2}
\min(\{m-1,\ell_2+\ell_3-(s+s')\})+1+\ell_1-1 &\leq \ell_2+\ell_3-(m-1)+\ell_1=m+1\leq v-1
\end{align}
Then by using Theorem~{\normalfont\ref{setsAndRelabeling}} (since $(V,V(P')\cap V,V(Q')\cap V)$ is $(0,\ell_1-2)$-good), the vertices of $\mathcal{P}_1$ can be relabeled so that $V(Q_1)\cap V(Q')=\{x,y\}$. Thus there exists an $m$-cycle decomposition $\mathcal{P}'\cup \mathcal{P}_1\cup \{E(P'\cup P_1),E(Q'\cup Q_1)\}$ of $(\l+\mu)K_{v+u}-\l K_v$.

 \textbf{Case 2:} Suppose that both $\l+\mu$ and $\l$ are odd.

Let $\infty\in V$. Let $G_1=\mu K_v$ with vertex set $V$, $G_2=(\l+\mu)K_{v-1,u}$ with vertex set $(V\setminus\{\infty\})\cup U$, $G_3=(\l+\mu)K_{u+1}$ with vertex set $U\cup\{\infty\}$, and $G'=G_1\cup G_2$. 
Let $\ell_1,\ell_3\in\{0,4,6,8,\ldots,m-2,m+2\}$ such that $\ell_1\equiv |E(G_1)|\pmod{m}$ and $\ell_3\equiv |E(G_3)|\pmod{m}$. 
Let $\ell_2\in\{4,6,8,\ldots,m-4,m-2,m,m+2\}$ such that $\ell_2\equiv |E(G_2)|\pmod{m}$. Since $\l+\mu$ and $\l$ are odd, $v-1$ and $\mu$ are even, so by Theorem~{\normalfont\ref{multiLKn}}, there exists an $m$-cycle packing $\mathcal{P}_1$ of $G_1$ and $m$-cycle packing $\mathcal{P}_3$ of $G_3$ with a leave $L_1$ and $L_3$ that is a single cycle of length $\ell_1$ and length $\ell_3$ respectively. Note that if $\ell_1=0$ or $\ell_3=0$, then the leave is considered empty.

If $\ell_1=\ell_3=0$, then $\ell_2=m$ since $\ell_1+\ell_2+\ell_3\equiv 0\pmod{m}$. As in Case 1, if this were the case, we could form the appropriate $m$-cycle decomposition by using Theorems~{\normalfont\ref{multiLKn}} and {\normalfont\ref{bipartiteMaxPacking}}. 

Let $r=\frac{m}{2}$ if $\frac{m}{2}$ is even and let $r=\frac{m}{2}-1$ if $\frac{m}{2}$ is odd. Define $e$, $p_1$, $q_1$, $p_2$, $q_2$, $p_3$, $q_3$, $p_4$, and $q_4$ as in Table~{\normalfont\ref{easier}} depending on $\ell_1$, $\ell_2$, and $\ell_3$.

Suppose that $\ell_1$, $\ell_2$, and $\ell_3$ are defined as in any one row of Table~{\normalfont\ref{easier}} and suppose that if $\ell_1\neq 0$ and $\ell_1+\ell_2+\ell_3=2m$, then $\ell_2+\ell_3<m+2$. As said in Case 1, if $\ell_1\neq 0$, $\ell_3\neq 0$, $\ell_1+\ell_2+\ell_3=2m$, and $\ell_2+\ell_3<m+2$, then $\ell_1=m+2$ and $\ell_2+\ell_3=m-2$.
Recall that $u\geq m+2$ and $v\geq m+2$. Notice that since $G_2$ has multiplicity greater than $2$, in each case when $e=m+\ell_2$, $\ell_2\leq m-4$ and thus $e\leq 2m-4$; so the requirement in Theorem~{\normalfont\ref{multiGraphPaths}} that the size of the leave is at most $2m-4$ when $v-1\geq  m+1$ is satisfied. (Note that the size of the leave for the packing in Theorem~{\normalfont\ref{multiGraphPaths}} is at most $2m-4$ instead of $2m-2$ since we are using this theorem with $v-1$ vertices in one of the parts instead of $v$ vertices.)
Then by Theorem~{\normalfont\ref{multiGraphPaths}}, there exists an $m$-cycle packing $\mathcal{P}_2$ of $G_2$ with a leave that is a cycle of size $e$.
It is clear that $L_3$ can be decomposed into a $p_3$-path $P_3$ and a $q_3$-path $Q_3$ with both end vertices in $U$. 
Since $|(V(P_3)\cap U)\cap (V(Q_3)\cap U)|\leq \min(\{p_3,q_3\})$, by using Theorem~{\normalfont\ref{setsAndRelabeling}} it follows that $(U,V(P_3)\cap U, V(Q_3)\cap U)$ is $(\frac{p_2-2}{2},\frac{q_2-2}{2})$-good. 
If $\ell_1=0$, by  Lemma~{\normalfont\ref{lemma6.1}}(a) (with $p=p_3$, $q=q_3$, $p'=p_2$, and $q'=q_2$), there exists an $m$-cycle decomposition of $G'$ and so along with the packing $\mathcal{P}_1$, we attain the required decomposition of $(\l+\mu)K_{v+u}-\l K_v$. Otherwise, after we join $G_2$ and $G_3$ to form $G'$, we must join the leave in $G'$ and the leave in $G_1$ together. 
Thus by Lemma~{\normalfont\ref{lemma6.1}}(b) (with $p=p_3$, $q=q_3$, $p'=p_2$, and $q'=q_2$), there exists an $m$-cycle packing $\mathcal{P}'$ of $G'$ with a leave that can be decomposed into a $p_4$-path $P'$ and a $q_4$-path $Q'$. 
Since $p_1,q_1<v$, it is clear that $L_1$ can be decomposed into a $p_1$-path $P_1$ with end vertices $x$ and $y$ in $V\setminus\{\infty\}$ and a $q_1$-path $Q_1$ so that if $\infty\in V(P')$ then $\infty\not\in V(P_1)$ or if $\infty\in V(Q')$ then $\infty\not\in V(Q_1)$.
Note that since $|((V(P_1)\cap (V\setminus\{\infty\}))\cap (V(Q_1)\cap (V\setminus\{\infty\}))|\leq \min(\{p_1,q_1\})$, it follows that Equation~({\normalfont\ref{ScapTArgument}}) holds (since $p_1$, $q_1$, $p_4$, and $q_4$ are the same as in Case 1). 
Then by using Theorem~{\normalfont\ref{setsAndRelabeling}} (since $(V\setminus\{\infty\}, V(P_1)\cap (V\setminus \{\infty\}), V(Q_1)\cap (V\setminus \{\infty\}))$ is $(\left\lfloor\frac{p_4-2}{2}\right\rfloor,\left\lfloor\frac{q_4-2}{2}\right\rfloor)$-good), the vertices of $\mathcal{P}_1$ can be relabeled 
so that $V(Q_1)\cap V(Q')=V(P_1)\cap V(P')=\{x,y\}$. Thus there exists an $m$-cycle decomposition $\mathcal{P}'\cup \mathcal{P}_1\cup \{E(P'\cup P_1),E(Q'\cup Q_1)\}$ of $(\l+\mu)K_{v+u}-\l K_v$. 

Now suppose $\ell_1\neq 0$, $\ell_1+\ell_2+\ell_3=2m$, and $\ell_2+\ell_3\geq m+2$, so $\ell_2+\ell_3\leq 2m-4$. 
By Theorem~{\normalfont\ref{bipartiteMaxPacking}}, there exists an $m$-cycle packing $\mathcal{P}_2$ of $G_2$ with a leave that is a cycle of size $\ell_2$.
Let $s$ and $s'$ be positive integers such that $s+s'=m-1$, $\ell_3-s\geq 1$, $\ell_2-s'\geq 2$, $s$ and $s'$ are positive integers, and $s'$ is even. Such integers exist since $\ell_2+\ell_3\geq m+2$ and $\ell_2,\ell_3>0$. It is clear that $L_3$ can be decomposed into an $s$-path $P_3$ and an $(\ell_3-s)$-path $Q_3$.  
Since $|V(P_3)\cap U|=s$, $|V(Q_3)\cap U|=\ell_3-s$ and hence $|(V(P_3)\cap U)\cap (V(Q_3)\cap U)|\leq \frac{\ell_3}{2}$ and $\frac{\ell_3}{2}+\frac{s'-2}{2}+\frac{\ell_2-s'-2}{2}\leq \frac{2m}{2}=m$, it follows by Theorem~{\normalfont\ref{setsAndRelabeling}} that $(U,V(P_3)\cap U, V(Q_3)\cap U)$ is $(\frac{s'-2}{2},\frac{\ell_2-s'-2}{2})$-good. 
Thus by Lemma~{\normalfont\ref{lemma6.1}}(b) (with $p=s$, $q=\ell_3-s$, $p'=s'$, and $q'=\ell_2-s'$), there exists an $m$-cycle packing $\mathcal{P}'$ of $G'$ with a leave that can be decomposed into an $(m-1)$-path $P'$ and an $(\ell_2+\ell_3-m+1)$-path $Q'$ with both end vertices in $V\setminus\{\infty\}$. 
Since $\ell_1-1<v$, it is clear that $L_1$ can be decomposed into a $1$-path $P_1=[x,y]$ with end vertices $x,y\in V\setminus\{\infty\}$ and an $(\ell_1-1)$-path $Q_1$ so that if $\infty\in V(P')$ then $\infty\not\in V(P_1)$ or if $\infty\in V(Q')$ then $\infty\not\in V(Q_1)$.
Note that since $|((V(P')\cap (V\setminus\{\infty\}))\cap (V(Q')\cap (V\setminus\{\infty\}))|\leq \min(\{m-1,\ell_2+\ell_3-s-s'\})$, it follows that Equation~({\normalfont\ref{ScapTArgument2}}) holds (since $p_1$, $q_1$, $p_4$, and $q_4$ are the same as in Case 1). 
Then by using Theorem~{\normalfont\ref{setsAndRelabeling}} (since $(V\setminus\{\infty\},V(P')\cap(V\setminus\{\infty\}),V(Q'\cap(V\setminus\{\infty\}))$ is $(0,\ell_1-2)$-good), the vertices of $\mathcal{P}_1$ can be relabeled so that $V(Q_1)\cap V(Q')=\{x,y\}$ (do not relabel to ensure that we retain the properties for $P_1$ and $Q_1$ described in the decomposition of $L_1$). Thus there exists an $m$-cycle decomposition $\mathcal{P}'\cup \mathcal{P}_1\cup \{E(P'\cup P_1),E(Q'\cup Q_1)\}$ of $(\l+\mu)K_{v+u}-\l K_v$. 

\end{proof}

\small
\bibliographystyle{amsplain}	
\bibliography{vdec2}		

\end{document}